\documentclass[11pt]{amsart}

\usepackage[utf8]{inputenc}
\usepackage{enumerate}
\usepackage{amsfonts,amsmath}
\usepackage{amssymb}
\usepackage{latexsym}
\usepackage{amsthm,amscd}
\usepackage{setspace} 
\usepackage{hyperref}
\usepackage{cite}
\usepackage{enumerate}
\usepackage{stmaryrd}
\usepackage{stackrel}

\newtheorem{Theorem}{Theorem}[section]
 \newtheorem{cor}[Theorem]{Corollary}
 \newtheorem{Lemma}[Theorem]{Lemma}
 
 \theoremstyle{definition}
 
 \theoremstyle{remark}
 \newtheorem{Remark}{Remark}
 \newtheorem{example}{Example}
 \numberwithin{equation}{section}

\def\la{{\lambda}}

\def\Z{{\mathbb Z}}
\def\Q{{\mathbb Q}}

\def\ID2{$(\text{ID}_2)$}

\begin{document}

\title[Idempotent factorizations of matrices over integer rings]{Idempotent factorizations of singular $2\times 2$ matrices over quadratic integer rings}

\author{Laura Cossu}

\address{Laura Cossu \\Dipartimento di Matematica ``Tullio Levi-Civita''\\
 Via Trieste 63\\  35121 Padova, Italy}

\email{laura.cossu@unipd.it} 

%\author{Giulio Peruginelli}
%
%\address{Giulio Peruginelli \\Dipartimento di Matematica ``Tullio Levi-Civita''\\ Via Trieste 63\\  35121 Padova, Italy}

\author{Paolo Zanardo}

\address{Paolo Zanardo \\Dipartimento di Matematica ``Tullio Levi-Civita''\\ Via Trieste 63\\  35121 Padova, Italy}

\email{pzanardo@math.unipd.it}

 
 
\thanks{This research was supported by Dipartimento di Matematica ``Tullio Levi-Civita'' Università di Padova, under Progetto BIRD 2017 - Assegni SID (ex Junior) and Progetto DOR1690814 ``Anelli e categorie di moduli''. The first author is a member of the Gruppo Nazionale per le Strutture Algebriche, Geometriche e le loro Applicazioni (GNSAGA) of the Istituto Nazionale di Alta Matematica (INdAM).
}

\begin{abstract}
Let $D$ be the ring of integers of a quadratic number field $\Q[\sqrt{d}]$. We study the factorizations of $2 \times 2$ matrices over $D$ into idempotent factors. When $d < 0$ there exist singular matrices that do not admit idempotent factorizations, due to results by Cohn (1965) and by the authors (2019). We mainly investigate the case $d > 0$. We employ Vaser\v{s}te\u{\i}n's result (1972) that $SL_2(D)$ is generated by elementary matrices, to prove that any $2 \times 2$ matrix with either a null row or a null column is a product of idempotents. As a consequence, every column-row matrix admits idempotent factorizations. 
\end{abstract}

\maketitle

\section*{Introduction}
In 1967 J. A. Erdos \cite{Erdos} proved that every square matrix with entries in a field that is singular (i.e., with zero determinant) admits {\it idempotent factorizations}, i.e., it can be written as a product of idempotent matrices. His work originated the problem of characterising integral domains $R$ that satisfy the property
ID$_n$: every singular $n \times n$ matrix with entries in $R$ admits idempotent factorizations. We recall the papers by Laffey \cite{Laff1} (1983), Hannah and O'Meara \cite{HannahOmeara} (1989), Fountain \cite{Fount} (1991). 

The  related problem of factorizing invertible matrices into products of elementary factors has been a prominent subject of research over the years (see \cite{Cohn, Omeara, BMS, Vas, Liehl, MRS}).
%In this paper we focus on singular dimension $2$ matrices over quadratic integer rings. 
Following Cohn \cite{Cohn}, we say that $R$ satisfies property GE$_n$ if every invertible $n \times n$ matrix over $R$ can be written as a product of elementary matrices.

Actually, when $R$ is a B\'ezout domain (i.e., every finitely generated ideal is principal) it is not restrictive to confine ourselves to $2 \times 2$ matrices, since, by \cite{Laff1}, \cite{Kap_el_div}, property ID$_2$ and GE$_2$ imply properties ID$_n$ and GE$_n$ for every $n > 0$, respectively.
%Cohn proved that every imaginary quadratic integer ring which is not Euclidean does not satisfy property GE$_2$. Using our results in \cite{CZ_Pr\"ufer}, we observe that these rings do not even satisfy ID$_2$ thus verifying the conjecture above whenever they are not PIDs.

For $R$ a B\'ezout domain, Ruitenburg \cite{Ruit} (1993) proved the crucial equivalence of the two properties, namely: $R$ satisfies ID$_n$ for every $n\geq 2$ if and only if it satisfies GE$_n$ for every $n\geq 2$. 

Note that fields, Euclidean domains, valuation domains verify the above factorizations properties, but not every Principal Ideal Domain does. Classes of non-Euclidean PIDs not satisfying property GE$_2$ (hence neither ID$_2$, by \cite{Ruit}) may be found in \cite{Cohn} and \cite{CZZ}.
%Another motivation for the interest aroused by idempotent factorizations of matrices is their relationship with the study of generalisations of Euclidean algorithms. A B\'ezout domain $R$ admits a {\it weak algorithm} (we refer to Section 1 for a precise definition) if and only if all singular $n\times n$ matrices over $R$ can be written as products idempotent factors \cite{Omeara,Ruit}. 

The factorization of matrices into idempotent factors has been extensively studied also in the non-commutative setting, see, for instance, \cite{HannahOmeara, AJLL, AJL_1,FL}.

The above results motivated Salce and Zanardo \cite{SalZan} (2014) to conjecture that an integral domain satisfying ID$_2$ must be a B\'ezout domain. The classes of unique factorization domains, projective-free domains, and PRINC domains (defined in \cite{SalZan}, and later studied in \cite{PerSalZan,CZ_PRINC}) verify the conjecture (cf.\cite{SalZan, BhasRao}). The authors proved in \cite{CZ} that every domain satisfying ID$_2$ must be a Pr\"ufer domain (i.e., a domain in which every finitely generated ideal is invertible) that satisfies property GE$_2$. They showed that a large class of coordinate rings of plane curves and the ring of integer-valued polynomials Int($\mathbb{Z}$) verify the conjecture. Idempotent factorizations of $2 \times 2$ matrices over special Pr\"ufer domains are also investigated in \cite{CZ_Dress}.

In this framework, it is natural to investigate idempotent matrix factorizations over rings of integers in number fields. Deep results solved the corresponding problem of factorizing invertible matrices over rings of integers as products of elementary matrices. Indeed, for $R$ a ring of integers, Bass, Milnor and Serre, \cite{BMS} (1967), proved that $R$ satisfies GE$_n$ for every $n \ge 3$, and Vaser\v{s}te\u{\i}n, \cite{Vas} (1972) (see also \cite{Liehl}), proved that it also  satisfies property GE$_2$.

In this paper we investigate idempotent factorizations of singular $2 \times 2$ matrices over quadratic integer rings. We remark that Cohn \cite{Cohn} proved that every imaginary quadratic integer ring which is not Euclidean does not satisfy property GE$_2$. Then, by the results in \cite{CZ}, these rings do not even satisfy ID$_2$ (see Theorem \ref{Cohn}). For this reason we mostly focus on real quadratic integer rings.

In the first section we prove some preliminary results, whose main purpose is to allow us to reduce to simpler matrices in the process of factorization. In the second section we prove our main result, namely, every matrix $\begin{pmatrix}
x & y\\
0 & 0
\end{pmatrix}$, $x, y$ in a real quadratic integer ring $D$, admits idempotent factorizations. For the proof, we employ the above recalled result by Vaser\v{s}te\u{\i}n \cite{Vas} that $D$ satisfies GE$_2$. As a consequence of the main result, in the final section we show that also the column-row matrices admit idempotent factorizations. Moreover, we provide examples of idempotent factorizations of non-column-row matrices over $D = \Z [\sqrt{10}]$. These examples essentially illustrate how challenging is to find idempotent factorizations, in general.

It is worth noting that our results seem to weaken the above recalled conjecture. However, in spite of the large families of singular $2 \times 2$ matrices we are able to factorize, the existence of a real quadratic integer ring, not a PID, satisfying property ID$_2$, remains an open question.

The authors wish to thank Giulio Peruginelli for several useful discussions.

%
%Cohn and $d<0$.
%
%Use of result on elementary factorization, interesting! (bass-milnor-serre, vasershtein)
%
%Is the conjecture false? GE2 and ID2 are equivalent over Prufer? (examples in CZ don't have GE2)

\section{Notation and preliminary results}

Let $R$ be an integral domain. We will denote by $R^{\times}$ the group of units of $R$ and by $M_n(R)$ the ring of $n\times n$ matrices with entries in $R$. A matrix $\mathbf{M}\in M_n(R)$ is said to be {\it singular} if $\det\mathbf{M}=0$, {\it invertible} if $\det\mathbf{M}\in R^{\times}$. As usual, the general and special linear group over $R$ are denoted by $GL_n(R)$ and $SL_n(R)$.  A matrix $\mathbf{T}\in M_n(R)$ is {\it idempotent} if $\mathbf{T}^2=\mathbf{T}$. Clearly, every non-identity idempotent matrix over an integral domain is singular. 

We remark that every matrix similar to an idempotent matrix is also idempotent, hence a singular matrix $\mathbf{S}$ has an idempotent factorization if and only if any matrix similar to $\mathbf{S}$ is a product of idempotents. The property of being idempotent or product of idempotents is also preserved by transposition. 

We will denote by ID$_n(R)$ the set of singular matrices over $R$ that admit idempotent factorizations. So $R$ verifies the property ID$_n$ of the introduction if and only if ID$_n(R)$ coincides with the set of $n\times n$ singular matrices over $R$. Analogously, we will denote as GE$_n(R)$ the set of invertible matrices in $M_n(R)$ that can be written as products of elementary matrices and $R$ verifies the property GE$_n$ if and only if GE$_n(R)=GL_n(R)$.

In what follows we will focus on $2 \times 2$ matrices.

An easy computation shows that a singular nonzero matrix $\begin{pmatrix}
a & b\\
c & d
\end{pmatrix}\in M_2(R)$ is idempotent if and only if $d = 1 - a$ (cf. \cite{SalZan}). A pair of elements $a, b\in R$ is said to be an {\it idempotent pair} if $(a \ b)$ is the first row of an idempotent matrix.

We will mainly consider matrices of the form $\begin{pmatrix}
x & y\\
0 & 0
\end{pmatrix}$, $x, y \in R$. Hence, to simplify the notation, we will denote by $[ x \ y]$ the singular matrix $\begin{pmatrix}
x & y\\
0 & 0
\end{pmatrix}$.

\begin{Lemma} \label{scambio}
Let $R$ be a integral domain, and $x, y \in R$. Then $ [x \ y]\in {\rm ID}_2(R)$ if and only if $[y \ x]\in {\rm ID}_2(R)$.
\end{Lemma}

\begin{proof} Since \[\begin{pmatrix}
0 & 1\\
1 & 0
\end{pmatrix}[x \ y] \begin{pmatrix}
0 & 1\\
1 & 0
\end{pmatrix} = \begin{pmatrix}
0 & 0\\
y & x
\end{pmatrix},\]

if $[x \ y]$ is a product of idempotent matrices, then $\begin{pmatrix}
0 & 0\\
y & x
\end{pmatrix}$ is also a product of idempotents. Moreover
\[
[y \ x] = [ 1 \ 1]\, \begin{pmatrix}
0 & 0\\
y & x
\end{pmatrix},
\]
hence $[y \ x]$ is a product of idempotents. The argument is reversible.
\end{proof}

\medskip

Following the notation in \cite{SalZan}, we say that two nonzero elements $a, b$ of an integral domain $R$ admit a {\it weak (Euclidean) algorithm} if there exists a sequence of divisions
$r_i = q_{i +1}r _{i+1} + r_{i+2}$,  with $r_i , q_i \in R$, $- 1 \le i \le n - 2$, such that $a = r_{-1}$, $b = r_0$ and $r_n = 0$.

We will need the two lemmas that follow. The first one is based on \cite[Th.~14.3]{Omeara}.

\begin{Lemma}\label{OMeara_based}
Let $R$ be an integral domain. If $\begin{pmatrix}
x & y\\
z & t
\end{pmatrix}\in SL_2(R)$ lies in GE$_2(R)$ (i.e., it can be written as a product of elementary matrices), then $x,y\in R$ admit a weak algorithm.
\end{Lemma}
\begin{proof}
By assumption, we have $q_0,\dots,q_k\in R$ such that
\[\begin{pmatrix}
x & z\\
y & t
\end{pmatrix}=\begin{pmatrix}
1 & q_0\\
0 & 1
\end{pmatrix}\begin{pmatrix}
1& 0\\
q_1 & 1
\end{pmatrix}\dots \begin{pmatrix}
1 & 0\\
q_k & 1
\end{pmatrix}.\]
We observe that, in order to get the above factorization, it might be necessary to introduce some terms with $q_i=0$. Define 
\begin{align*}
&r_{1}=x-y q_0,\\
&r_{2}=y-r_1 q_1,\\
&r_{i+2}=r_i-r_{i+1}q_{i+1}\qquad\text{for}\qquad i=1,\dots,k-1.
\end{align*}
Then, premultiplying $\begin{pmatrix}
x & z\\
y & t
\end{pmatrix}$ by $\begin{pmatrix}
1 & -q_0\\
0 & 1
\end{pmatrix}$, then by $\begin{pmatrix}
1 & 0\\
-q_1 & 1
\end{pmatrix}$ and so on, we get from the matrix equation above that 
\[\begin{pmatrix}
r_k & *\\
r_{k+1} & *
\end{pmatrix}=\begin{pmatrix}
1 & 0\\
0 & 1
\end{pmatrix}.\]
Therefore there exists a finite sequence of relations of the form
\[r_i=q_{i+1}r_{i+1}+r_{i+2},\quad\text{with }i=-1,\dots,k-1\]
such that $r_{-1}=x$, $r_0=y$ and $r_{k+1}=0$. So $R$ admits a weak Euclidean algorithm.
\end{proof}

The next lemma follows from \cite[Th.~6]{AJLL}. For the reader's convenience we give the direct proof, inspired by that of \cite[Th.~6.2]{SalZan}. 

\begin{Lemma}\label{we_ID2}
Let $R$ be an integral domain. If $x,y\in R$ admit a weak algorithm, then $[x \ y] \in {\rm ID}_2(R)$.
\end{Lemma}

\begin{proof}
 By assumption, there exists a finite sequence of equalities $r_i=q_{i+1}r_{i+1}+r_{i+2}$ with $i=-1,\dots,n-2$ such that $r_{-1}=x$, $r_{0}=y$ and $r_n=0$. At the first step, we get $x=y q_0+r_1$ and we get the following relation of similarity:
\[\begin{pmatrix}
1 & 0\\
q_0 & 1
\end{pmatrix}[x \ y]\begin{pmatrix}
1 & 0\\
-q_0 & 1
\end{pmatrix}=\begin{pmatrix}
r_1 & b\\
q_0r_1 & q_0 y
\end{pmatrix}.\]
Therefore, to verify that $[x \ y] \in {\rm ID}_2(R)$, it suffices to show that $\begin{pmatrix}
r_1 & b\\
q_0r_1 & q_0 y
\end{pmatrix}$ is a product of idempotents. Since 
\[\begin{pmatrix}
r_1 & y\\
q_0 r_1 & q_0 y
\end{pmatrix}=\begin{pmatrix}
1 & 0\\
q_0 & 0
\end{pmatrix}[r_1 \ y]\]
 it is enough to show that $[r_1 \ y]\in {\rm ID}_2(R)$. At the second step of the algorithm we get $y=q_1r_1+r_2$ and 
\[\begin{pmatrix}
1 & q_1\\
0 & 1
\end{pmatrix}[r_1 \ y]\begin{pmatrix}
1 & -q_1\\
0 & 1
\end{pmatrix}=[r_1 \ r_2].\]
Hence $[r_1 \ y]$ is similar to $[r_1 \ r_2]$ and it suffices to prove that this latter matrix is product of idempotents. Repeating this procedure, after $n-2$ steps, it remains to prove that either $[r_{n} \ r_{n-1}]\in {\rm ID}_2(R)$ or $[r_{n-1} \ r_{n}]\in {\rm ID}_2(R)$. Since $r_{n}=0$, we get in the first case
\[[0 \ r_{n-1}]=[1 \ 0]\begin{pmatrix}
0 & r_{n-1}\\
0 & 1
\end{pmatrix},\]
and in the second case
\[[r_{n-1} \ 0]=[1 \ -1]\begin{pmatrix}
1 & 0\\
1-r_{n-1} & 0
\end{pmatrix}.\]
Finally, we conclude that $[x \ y]\in ID_2(R)$. 
\end{proof}

We refer to \cite{NZM} for the notions in Algebraic Number Theory that we will need in the remainder of the paper.

Let $d$ be a square-free integer and let $k$ be the real quadratic number field $k=\Q[\sqrt{d}]$. Let $D$ be the {\it ring of integers of $k$}, i.e., the integral closure of $\Z$ in $k$. Recall that
\[D=\begin{cases}\Z[\sqrt{d}]=\{\alpha+\beta\sqrt{d}\,|,\alpha,\beta \in \Z\}\quad\text{ if }d\equiv 2,3\mod 4\\
\Z\left[\frac{1+\sqrt{d}}{2}\right]\{\frac{\alpha+\beta\sqrt{d}}{2}\,|,\alpha,\beta \in \Z, \alpha\equiv \beta\mod 2\}\quad\text{ if }d\equiv 1\mod 4
\end{cases}.\]

\begin{Lemma}\label{x_integer}
Let $d>0$ be square-free integer, $D$ the ring of integers of $k=\Q[\sqrt{d}]$ and $x,y$ two non-zero elements in $D$. Then there exist $\alpha\in\Z$ and $w\in D$ such that $[x \ y]\in {\rm ID}_2(D)$ if and only if $[\alpha \ w]\in {\rm ID}_2(D)$.
\end{Lemma}

\begin{proof}
We assume that $d\equiv 2,3$ modulo $4$, the case $d\equiv 1$ modulo $4$ being wholly analogous. Set $x=x_1+x_2 \sqrt{d}$ and $y=y_1+ y_2 \sqrt{d}$ with $x_1,x_2,y_1,y_2 \in\Z$. By Lemma \ref{scambio} the result is immediate if either $x\in\Z$ or $y\in\Z$, hence we may assume $x_2 \ne 0 \ne y_2$. Apply the Euclidean algorithm to the couple of integers $x_2,y_2$. There exists a finite sequence of $n\geq 1$ divisions with last reminder 0:
\begin{align*}
&y_2=q_1 x_2+r_1\\
&x_2=q_2 r_1+r_2\\
%r_1=q_3 r_2+r_3\\
&\dots\\
&r_{n-2}=q_{n}r_{n-1}
\end{align*}
Note that
$$
\begin{pmatrix}
1 & - q_1\\
0 & 1
\end{pmatrix} [x_1+x_2 \sqrt{d} \quad y_1+y_2 \sqrt{d}]\begin{pmatrix}
1 & q_1\\
0 & 1
\end{pmatrix} = [x_1+x_2 \sqrt{d}  \quad w_1+r_1\sqrt{d}],
$$ 
where $w_1=y_1-q_1 x_1$. Therefore 
\[ [x \ y]\in {\rm ID}_2(D)\Leftrightarrow [x_1+x_2 \sqrt{d}  \quad w_1+r_1 \sqrt{d}]\in {\rm ID}_2(D).\] 

Moreover, this happens if and only if $[w_1+r_1 \sqrt{d} \quad x_1+x_2 \sqrt{d} ]\in {\rm ID}_2(D)$ by Lemma \ref{scambio}.
Iterating the process, at the $(n-1)$-th step, we get that 
 \[[x \ y]\in {\rm ID}_2(D)\Leftrightarrow [w_{n-1}+r_{n-1} \sqrt{d}  \quad w_{n-2}+r_{n-2} \sqrt{d}]\in {\rm ID}_2(D),\]
with $w_{j}=w_{j-2}-q_j w_{j-1}$ for $j=1,\dots,n$ (we assume $w_{-1}=y_1$ and $w_0=x_1$).

Since $[w_{n-1}+r_{n-1} \sqrt{d}  \quad w_{n-2}+r_{n-2} \sqrt{d}]$ is similar to  $[w_{n-1}+r_{n-1} \sqrt{d}  \quad w_{n}]$, we conclude that 
\[ [x \ y]\in {\rm ID}_2(D)\Leftrightarrow [
w \ z] \in {\rm ID}_2(D),\]
with $w=w_n\in\Z$ and $z= w_{n-1}+r_{n-1} \sqrt{d}$.
\end{proof}

The next result, essentially due to Cohn \cite{Cohn}, leads to confine our investigation to the case where $d > 0$.

\begin{Theorem}\label{Cohn}
Let $d < 0$ be a square-free integer, $D$ the imaginary ring of integers of $k=\Q[\sqrt{d}]$. Then $D$ does not satisfy property {ID}$_2$, except when $d = -1, -2, -3, -7, -11$, for which values $D$ is an Euclidean domain.
\end{Theorem}

\begin{proof}
By Theorem 6.1 of Cohn's paper \cite{Cohn}, $D$ does not satisfy property GE$_2$, for every $d < 0$ different from the values listed in the statement. By Proposition 3.4 in \cite{CZ}, such integer ring does not satisfy ID$_2$, as well. 
\end{proof}

In view of Theorem \ref{Cohn}, in the remainder of the paper we will focus on rings of integers of $\Q[\sqrt{d}]$, with $d > 0$ a square-free integer.

\begin{Remark}
In the proof of the above theorem, we used the fact that an arbitrary integral domain satisfying ID$_2$ satisfies GE$_2$, as well. We recall that the converse is not true. For instance, any local domain satisfies GE$_n$ for every $n > 0$, while it satisfies ID$_2$ only when it is a valuation domain (see Remark 1 of \cite{CZ}).
\end{Remark}

\section{Main results}

In what follows $D$ will denote the ring of integers of $\Q[\sqrt{d}]$, $d > 0$ a square-free integer. In this section we prove that all the matrices $[x \ y]$, $x, y \in D$, admit idempotent factorizations. 

\medskip

Let $k$ be a number field. Let $S$ be a finite subset of the set $V^k$ of the valuations of $k$ containing the set $V^k_{\infty}$ of archimedean valuations. The ring of $S$-integers $\mathcal{O}_{k,S}$ in $k$ is:
\[\mathcal{O}_{k,S}=\{a\in k^*\,|\, v(a)\geq 0 \text{ for all }v\in V^k\setminus S\}\cup \{0\}.\]

 Vaser\v{s}te\u{\i}n in \cite{Vas} (see also \cite{Liehl}) proved that, if the ring of $S$-integers $\mathcal{O}_{k,S}$ has infinitely many units, then the group $SL_2(\mathcal{O}_{k,S})$ is {\it boundedly} generated by elementary matrices. The main result in \cite{MRS} shows that every matrix in $SL_2(\mathcal{O}_{k,S})$, with $\mathcal{O}_{k,S}^{\times}$ infinite, is a product of at most $9$ elementary matrices.
 
 We remark that the real quadratic integer ring $D$ has infinitely many units, by a classical result in number theory (cf. \cite[Th.~9.23]{NZM}). Then Vaser\v{s}te\u{\i}n's result shows that $D$ satisfies property GE$_2$. From this fact we will derive properties of idempotent factorizations of singular matrices over $D$. 

The following preliminary theorem is valid for any ring of $S$-integers with infinitely many units. However, according to our purposes, we confine ourselves to real quadratic integer rings $D$. 

\begin{Theorem}\label{comaximal}
Let $D$ be the ring of integers of a real quadratic number field. If $x,y\in D$ generate a principal ideal of $D$, then $\begin{pmatrix}
x & y\\
0 & 0
\end{pmatrix}$ admits idempotent factorizations.
\end{Theorem}

\begin{proof}
It is not restrictive to assume that $x D+ y D= D$. In fact, if $x= d x'$ and $y= d y'$ with $d,x',y'\in D$ and $x' D+ y' D= D$, then $[x \ y]=[d \ 0] [x' \ y']$, where $[d \ 0]=[1 \ -1] [1 \quad 1-p]^T  \in {\rm ID}_2(D)$. Therefore there exist $z,t\in D$ such that $\begin{pmatrix}
x & y\\
z & t
\end{pmatrix}\in SL_2(D)$. Since $D$ satisfies property GE$_2$ \cite{Vas}, $x$ and $y$ admit a weak Euclidean algorithm by Lemma \ref{OMeara_based}, and we conclude by Lemma \ref{we_ID2}.
\end{proof}

The previous theorem is crucial to show that any matrix $[x \ y]$ ($x, y \in D$) admits idempotent factorizations.

\begin{Theorem}\label{generalcase}
Let $D$ be the ring of integers of the real quadratic number field $\Q[\sqrt{d}]$, $d > 0$ a square-free integer.
Then every matrix over $D$ of the form  $\begin{pmatrix}
x & y\\
0 & 0
\end{pmatrix}$ admits idempotent factorizations. 
\end{Theorem}

\begin{proof}
By Lemma \ref{x_integer} we may assume, without loss of generality, that $x \in \Z$. If $x, y$ have a common non-unit factor in $D$, say $z$, we get $[x \ y] = [z \ 0] [ x/z \ y/z]$, where $[z \ 0] \in {\rm ID}_2(D)$. Hence $[x \ z] \in {\rm ID}_2(D)$ as soon as $[x/z \ y/z] \in {\rm ID}_2(D)$. It follows that we may also assume that $x, y$ have no common non-unit factors in $D$. Moreover, by Theorem \ref{comaximal}, $[ x \ y]$ is a product of idempotent matrices if the ideal $xD + yD$ is principal. So we may assume that $xD + yD \ne D$, hence, in particular, $m = \gcd(x, ||y||) \ne 1$ ($||y|| = y \bar{y}$ is the norm of $y$). 

We proceed by steps.

\medskip
\noindent
{\bf STEP 1}. Assume $\gcd(x, ||y||/m) = 1$, where $1 \ne m = \gcd (x, ||y||)$. Then $[x \ y ] \in {\rm ID}_2(D)$ 

\medskip

We want to show that there exist an idempotent pair $a, b$ in $D$, and comaximal elements $x', y'$ of $D$ such that
\[ [x \ y]= [x' \ y']\begin{pmatrix}
a & b \\
c & 1 - a
\end{pmatrix}\,.\]

Let us set $||y||/m = \la$, and take $a', t \in \Z$ such that $a'x + t \la = 1$.
In the above notation, set
\[
\begin{pmatrix}
a & b \\
c & 1 - a
\end{pmatrix} =
\begin{pmatrix}
xa' & y a' \\
t \bar{y} x/m & 1 - xa' 
\end{pmatrix} \, .
\]
Note that  $1 - x a' = t \la$, and $xa' t \la - ya't \bar{y}x/m = 0$, hence the above matrix is idempotent.

Let us define $x' = x - t \la$, $y' = y(1 + a')$
Then we get 
\[
a x' + c y' = xa'(x - t \la) + t \bar{y}( x/m) y(1+a') = x(a'x + t \la) = x.
\]
Since the involved matrices are singular, we also get $ b x' + (1-a) y' = y$. Hence $x', y' \in D$ solve the above matrix equation. 

It remains to show that $(x', y') = D$. It suffices to verify that the integers $x - t \la$ and $||y||(1 + a')^2 = m \la (1 + a')^2$ are coprime. Both $\la$ and $t$ are coprime with $x$, hence they are coprime with $x'$, as well. Moreover 
\[
a'(x - t \la) + t \la(1 + a') = 1
\]
shows that $x'$ is coprime with $1 + a'$.
The desired conclusion follows.

Since $[ x' \ y'] \in {\rm ID}_2(D)$ by Theorem \ref{comaximal}, we immediately derive that $[x \ y] \in {\rm ID}_2(D)$, as well. 

\medskip
\noindent
{\bf STEP 2:} Let $d \equiv 2,3$ modulo $4$, so that $D = \Z[\sqrt{d}]$, and assume that $\gcd(x, ||y||/m) = s \ne 1$. Then $[x \ y] \in {\rm ID}_2(D)$.

\medskip
 
Say $y = y_1 + y_2\sqrt{d}$, $y_1, y_2 \in \Z$. Note that $(x, y_1, y_2) = 1$, since, by assumption, $x$ and $y$ have no common non-unit factors. We write $x = mx_0$, $||y|| = m \la$.
 For any prime number $p$, we consider the $p$-adic valuation on $\Q$, denoted by $v_p$. Under the present circumstances the following facts hold. 

\medskip
\noindent
{\bf FACT 1.} In the above notation, we have:

(i) if the prime number $p$ divides $s$, then $v_p(x_0) = 0$, $v_p(m) > 0$ and $v_p(\la) > 0$; 

(ii) if $p$ divides $x$, $y_1$ and $||y||$, then $v_p(||y||) = 1$; 

(iii) if $2$ divides $x$, $||y||$, and $y_1 \notin 2 \Z$, then $v_2(\la) = 0$.

\begin{proof}
(i) The statement follows readily from the definitions. 

(ii) Assume that $p$ divides $x$, $y_1$ and $||y||$. Then $y_2 \notin p \Z$, hence $v_p(||y||) = v_p(y_1^2 - d y_2^2) \ge 1$ yields $v_p(d y_2^2) = v_p(d) \ge 1$. It follows that $v_p(d) = 1$, since $d$ is square-free, so that $v_p(y_1^2) \ge 2$ yields $v_p(||y||) = v_p(d) = 1$. 

(iii) Assume that $2$ divides $x$, $||y||$, but $y_1 \notin 2 \Z$. It follows that $d y_2^2 \notin 2 \Z$, as well. We get
 $$
 ||y|| = y_1^2 - d y_2^2 \equiv 1 - d \not\equiv 0 \quad {\rm mod} \ 4.
 $$
  We conclude that $2$ does not divide $||y||/m = \la$. 
\end{proof}

\medskip
\noindent
{\bf FACT 2.} In the above notation, we may choose $e \in \Z$ such that $v_q(\la + e x_0 2 y_1) = 0$ for every prime number $q$ such that $v_q(\la) = 0$ and $v_q(x_0) = 0$. Then $x$ is coprime with $||y + e x||/m$.

\begin{proof}
We firstly show how to get $e \in \Z$ as in the statement. Let $A_1$ be the set of the primes $q$ dividing $x$ such that $v_q(\la) = 0$, $v_q(x_0) = 0$, and $v_q(\la +  x_0 2 y_1) > 0$; let $A_2$ be the set of the primes $q'$ dividing $x$ such that $v_{q'}(\la) = 0$, $v_{q'}(x_0) = 0$, and $v_{q'}(\la + x_0 2 y_1) = 0$; let $A_3$ be the set of the other primes $p'$ dividing $x$. By the Chinese Remainder Theorem there exists $e \in \Z$ such that
$$
e \equiv 2 \quad {\rm mod} \ q  \quad {\rm when} \ q \in A_1,
$$
$$
e \equiv 1 \quad {\rm mod} \  q'  \quad {\rm when} \ q' \in A_2,
$$
$$
e \equiv 1 \quad {\rm mod} \  p \quad {\rm when} \ p \in A_3.
$$
 Since $v_q(\la) = 0$ and $v_q(\la +  x_0 2 y_1) > 0$ imply $v_q(x_02y_1) = 0$, it is readily seen that $v_q(\la + e x_0 2 y_1) = 0$ for every $q \in A_1 \cup A_2$.

By direct computation, we get
$$
||y + ex||/m =  \la + ex_0 2 y_1 + e^2 x_0^2m.
$$
Take a prime number $p$ dividing $x$. 

If $v_p(x) > v_p(||y||)$, then $v_p(||y||) = v_p(m)$, hence $v_p(x_0) > 0$ and $v_p(\la) = 0$. It follows that $v_p( \la + e x_0 2 y_1 + e^2 x_0^2m) = v_p(\la) = 0$.

If $v_p(||y||) > v_p(x)$, then $v_p(x) = v_p(m)$, hence $p$ divides $s$. By statement (i) of Fact 1, $v_p(x_0) = 0$ and $v_p(\la) > 0$. Moreover, by statements (ii) and (iii) of Fact 1, we get $p \ne 2$ and $y_1 \notin p \Z$. It follows that $v_p( \la + e x_0 2 y_1 + e^2 x_0^2m) = v_p(e x_0 2 y_1) = v_p(x_0) = 0$.

Finally, assume that $v_p(x) = v_p(||y||)$. It follows that $v_p(x) = v_p(m)$, hence $v_p(x_0) = 0 = v_p(\la)$. Under the present circumstances, we get $v_p( \la + e x_0 2 y_1 + e^2 x_0^2m) = v_p( \la + e x_0 2 y_1) = 0$, by our choice of $e \in \Z$.

We conclude that $v_p(|| y + e x||) = 0$ for every $p$ dividing $x$, hence $x$ is coprime with $||y + e x||/m$.
\end{proof}

Due to the above FACT 2, since $[x \ y]$ is similar to the matrix $[x \ y+ex]$, that admits idempotent factorizations by STEP 1, we deduce that also the matrix $[ x \ y]$ has idempotent factorizations.

\medskip
{\bf STEP 3.} Let $d \equiv 1$ modulo $4$, so that $D = \Z[(1 +\sqrt{d})/2]$, and assume that $\gcd(x, ||y||/m) = s \ne 1$. Then $[x \ y] \in {\rm ID}_2(D)$.

\medskip

Say $y = (y_1 + y_2 \sqrt{d})/2 \in D$, with $y_1, y_2 \in \Z$, $ y_1 - y_2 \in 2\Z$. As in the second step we get $(x, y_1, y_2) = 1$.

We distinguish two cases:

\medskip
\noindent
{\bf CASE (a).} $y \in 2D$, so that $x \notin 2 \Z$.

\medskip

For convenience, we change notation and write $y = y_1 + y_2 \sqrt{d}$, where $y_1, y_2 \in \Z$ and $(x, 2y_1, 2y_2) = 1$.

We get

\medskip
\noindent
{\bf FACT 1(a).} In the above notation, we have:

(i) if the prime number $p$ divides $s$, then $v_p(x_0) = 0$, $v_p(m) > 0$ and $v_p(\la) > 0$; 

(ii) if $p$ divides $x$, $y_1$ and $||y||$, then $v_p(||y||) = 1$; 

\medskip

The proof is identical to that of Fact 1 (i), (ii) of the second section. Note that it is not necessary to considerer (iii) of Fact 1, since $p \ne 2$.

\medskip
\noindent
{\bf FACT 2(a).} In the above notation, we may choose $e \in \Z$ such that $v_q(\la + e x_0 2 y_1) = 0$ for every prime number $q$ such that $v_q(\la) = 0$ and $v_q(x_0) = 0$. Then $x$ is coprime with $||y + e x||/m$.

\medskip

The proof is identical to that of Fact 2 of the second section. (Actually, it is slightly simpler, since $p \ne 2$, hence $v_p(2) = 0$ for every $p$ dividing $x$.)

\medskip
\noindent
{\bf CASE (b).} $y = (y_1 + y_2 \sqrt{d})/2 \notin 2D$, so that $y_1, y_2$ are both odd.

\medskip

In this case we get

\medskip
\noindent
{\bf FACT 1(b).} In the above notation, we have:

(i) if the prime number $p$ divides $s$, then $v_p(x_0) = 0$, $v_p(m) > 0$ and $v_p(\la) > 0$; 

(ii) if $p$ divides $x$, $y_1$ and $||y||$, then $v_p(||y||) = 1$ (note that $p \ne 2$).

\begin{proof}
(i) The proof is identical to that of Fact 1(i).

(ii) We have $||y|| = (y_1^2 - d y_2^2)/4$. Since $p \ne 2$, $y_2 \notin p\Z$, $v_p(y_1^2) \ge 2$ and $v_p(d) \le 1$ (because $d$ is square-free), we get 
$$
1 \le v_p(||y||) = v_p(y_1^2 - d y_2^2) - v_p(4) = v_p(d) \le 1.
$$
\end{proof}

\medskip
\noindent
{\bf FACT 2(b).} In the above notation, we may choose $e \in \Z$ such that $v_q(\la + e x_0  y_1) = 0$ for every prime number $q$ such that $v_q(\la) = 0$ and $v_q(x_0) = 0$. Then $x$ is coprime with $||y + e x||/m$.

\begin{proof}
Under the present circumstances, from $y = (y_1 + y_2 \sqrt{d})/2$ we get
$$
||y + ex||/m =  \la + ex_0 y_1 + e^2 x_0^2m.
$$
Then the proof follows the argument in Fact 2, replacing $\la + e x_0 2 y_1$ with $\la + e x_0  y_1$.
\end{proof}

Now we see that $[x \ y] \in {\rm ID}_2(D)$, as in the second step. .

\medskip
We have reached the desired conclusion, since the preceding steps cover all the possibilities. 
\end{proof}

\section{Column-row matrices and Examples}

Following the terminology in \cite{SalZan}, a matrix $\textbf{M}\in M_2(R)$, with $R$ any integral domain, is called \textit{column-row} if there exist $a,b,x,y\in R$ such that
\[\textbf{M}=[ x \ y]^T [a \ b]=\begin{pmatrix}
xa & xb\\
ya & yb
\end{pmatrix}.\]
An easy computation shows that if $\textbf{M}$ is a singular matrix in $M_2(R)$ and the ideal generated by the elements of its first row is principal, then $\textbf{M}$ is a column-row matrix \cite[Prop.~2.2]{SalZan}. 

\begin{cor}\label{col-row}
Let $D$ be the ring of integers over any real quadratic number field. Any column-row matrix over $D$ is a product of idempotent matrices over $D$. In particular, every singular matrix in $M_2(D)$ having at least one row or column whose elements generate a principal ideal is a product of idempotent matrices over $D$.
\end{cor}
\begin{proof}
The first part of the statement is a direct consequence of Theorem \ref{generalcase}. Let $\mathbf{M}=\begin{pmatrix}
x & y\\
z & t
\end{pmatrix}$ be a singular matrix over $D$. If $(x,y)$ is principal, then $\mathbf{M}\in {\rm ID}_2(D)$ since, as observed above, it is column-row. If $(x,z)$ is a principal ideal, then ${\bf M}^T$ is column-row and then product of idempotent matrices over $D$. We conclude since ${\bf M}\in {\rm ID}_2(D)$ if and only if ${\bf M}^T\in {\rm ID}_2(D)$. If $(z,t)$ is principal, the matrix ${\bf N}=\begin{pmatrix}
t & z\\
y & x
\end{pmatrix}\in{\rm ID}_2(D)$. Therefore ${\bf M}\in{\rm ID}_2(D)$ since $\begin{pmatrix}
0 & 1\\
1 & 0
\end{pmatrix}{\bf N}\begin{pmatrix}
0 & 1\\
1 & 0
\end{pmatrix}={\bf M}$. Finally if $(y,t)$ is a principal ideal of $D$, then ${\bf N}^T$ is column-row and we conclude using the previous arguments.
\end{proof}

\begin{example}[Family of non-column-row matrices]\label{ex}
Let $p$ be a prime integer which is irreducible but not prime in the real quadratic integer ring $D$. If $z$ is an element in $D$ such that $(p,z)$ is a non-principal ideal of $D$, the singular matrix $\begin{pmatrix}
p & z\\
\bar{z} & || z ||/p
\end{pmatrix}$ is not a column-row matrix over $D$. In fact if $\begin{pmatrix}
p & z\\
\bar{z} & || z ||/p
\end{pmatrix}=\begin{pmatrix}
xa & xb\\
ya & yb
\end{pmatrix}$, from $p$ irreducible it follows that either $x\in D^{\times}$ or $a\in D^{\times}$, hence $(p , z)$ is principal, impossible. 
\end{example}

A complete classification of $2\times 2$ singular matrices over $D$ is still an open issue. We can only say, by Corollary \ref{col-row}, that a singular matrix in $M_2(D)$ which is not column-row must have rows and columns whose elements generate non-principal ideals. The question whether these matrices all lie in ID$_2(D)$ remains unanswered. 

In this last section we will exhibit some examples of idempotent factorizations of non-column-row matrices over $\Z[\sqrt{10}]$.

As recalled in the introduction, property ID$_2$ and GE$_2$ are equivalent over B\'ezout domains, therefore every real quadratic integer ring which is a PID satisfies property ID$_2$ by Vaser\v{s}te\u{\i}n's result \cite{Vas}. What happens over a non B\'ezout domain is still not known in general. The simplest example of a real quadratic integer ring that is not a PID is $\Z[\sqrt{10}]$. Over this ring we can show explicit examples of idempotent factorizations of non-column-row $2\times 2$ matrices obtained through rough direct computations with the support of the quadratic Diophantine equations BCMATH online programs at \url{http://www.numbertheory.org/php/main_pell.html}. In what follows $D=\Z[\sqrt{10}]$.

We firstly consider matrices of the form $\begin{pmatrix}
p & z\\
\bar{z} & || z ||/p
\end{pmatrix}$ as in Example \ref{ex}, i.e. $p$ is a prime number irreducible in $D$, $z \in D \setminus pD$, and $||z|| \in pD$:
\begin{equation*}
\small
\begin{pmatrix}
3 & 1+\sqrt{10}\\
1-\sqrt{10} & -3
\end{pmatrix}=\begin{pmatrix}
2+2\sqrt{10} & 7+\sqrt{10}\\
-6  & -1-2\sqrt{10}
\end{pmatrix}\begin{pmatrix}
2-2\sqrt{10} & -6\\
7-\sqrt{10} & -1+2\sqrt{10}
\end{pmatrix}
\end{equation*}
\begin{equation*}
\small
\begin{pmatrix}
2 & \sqrt{10}\\
-\sqrt{10} & -5
\end{pmatrix}=\begin{pmatrix}
6+2\sqrt{10} & 4+\sqrt{10}\\
 -10-3\sqrt{10}  &  -5-2\sqrt{10}
\end{pmatrix}\begin{pmatrix}
6-2\sqrt{10} & -10+3\sqrt{10}\\
4-\sqrt{10} & -5+2\sqrt{10}
\end{pmatrix}
\end{equation*}

The above factorizations seem to suggest that a matrix $\begin{pmatrix}
p & z\\
\bar{z} & || z ||/p
\end{pmatrix}$ of that type factorizes as 
\begin{equation}\label{decomp}
\begin{pmatrix}
p & z\\
\bar{z} & || z ||/p
\end{pmatrix}=\begin{pmatrix}
\bar{a} & \bar{c}\\
\bar{b} & 1-\bar{a}
\end{pmatrix}\begin{pmatrix}
a & b\\
c & 1-a
\end{pmatrix},
\end{equation}
for a suitable idempotent pair $a, b$.

A factorization of the desired form can be found also for the matrix $\begin{pmatrix}
13 & 8+3\sqrt{10}\\
8-3\sqrt{10} & -2
\end{pmatrix}$, with $a=-2738487 - 865986 \sqrt{10}$, $b=-3683652 - 1164873 \sqrt{10}$, $c=2035838 + 643788 \sqrt{10}$. 

\medskip

Unfortunately, our calculations are strongly dependent on $p$ and their complexity appear to increase with its size. This makes hard to prove or disprove the validity of \ref{decomp} in general.
%\begin{equation*}
%\begin{split}
%\begin{pmatrix}
%2 & \sqrt{10}\\
%\sqrt{10} & 5
%\end{pmatrix}&=\begin{pmatrix}
%-44-14\sqrt{10} & 28+9\sqrt{10}\\
%-70-22\sqrt{10}  & 45+14\sqrt{10}
%\end{pmatrix}\begin{pmatrix}
%-44+14\sqrt{10} & 70-22\sqrt{10}\\
%-28+9\sqrt{10} & 45-14\sqrt{10}
%\end{pmatrix}=\\&=\begin{pmatrix}
%\bar{a} & -\bar{c}\\
%-\bar{b} & 1-\bar{a}
%\end{pmatrix}\begin{pmatrix}
%a & b\\
%c & 1-a
%\end{pmatrix}
%\end{split}
%\end{equation*}

However, we can 
show that the decomposition in \ref{decomp} does not extend to matrices where $p$ is replaced by a non-prime integer. For instance, consider ${\bf S}=\begin{pmatrix}
8 & 2\sqrt{10}\\
-2\sqrt{10} & -5
\end{pmatrix}\in M_2(D)$. Easy considerations on the norms of its entries show that ${\bf S}$ is not a column-row matrix. If we assume that $ S=\bf T \ \bf U$, where ${\bf T} \in M_2(D)$ and ${\bf U}=\begin{pmatrix}
a & b\\
c & 1-a
\end{pmatrix}$ is idempotent over $D$, we get $a/b=c/(1-a)=8/(2\sqrt{10})$. This is possible only if 
\begin{align*}
&a=20 k - 4 + 2 h \sqrt{10},\\
&b=5 h + (5 k - 1)\sqrt{10},\\
&c=-8 h + 2 (1 - 4 k) \sqrt{10},
\end{align*}
with $h,k \in\Z$. Moreover, ${\bf T} = {\bf U}^H =\begin{pmatrix}
\bar{a} & \bar{c}\\
\bar{b} & 1-\bar{a}
\end{pmatrix}$ only if there exist $h,k\in\Z$ such that 
\[{\bf S}=\begin{pmatrix}
8(3h^2-30k^2+20k-3) & 2\sqrt{10}(3h^2-30k^2+20k-3)\\
-2\sqrt{10}(3h^2-30k^2+20k-3) & -5(3h^2-30k^2+20k-3)
\end{pmatrix},\]
i.e., only if there exist $h,k\in\Z$ such that $3h^2-30k^2+20k-3=1$. This is impossible, as it is immediate to see from the corresponding equivalence modulo $5$. We conclude that ${\bf S} \ne {\bf U}^H {\bf U}$, for any idempotent matrix $\bf U$. 

Nonetheless, we have found a factorization of ${\bf S}$ as product of two idempotent matrices over $D$, namely

\begin{equation*}
\small
{\bf S}=\begin{pmatrix}
-4-2\sqrt{10} & -8-2\sqrt{10}\\
 5+\sqrt{10}  &  5+2\sqrt{10}
\end{pmatrix}\begin{pmatrix}
16-4\sqrt{10} & -10+4\sqrt{10}\\
16-6\sqrt{10} & -15+4\sqrt{10}
\end{pmatrix}.
\end{equation*}

\begin{Remark}
The idempotent factorizations of the non-column-row matrices above raise doubts on the validity of the conjecture  mentioned in the introduction. Note that the examples in \cite{CZ} of Pr\"ufer non-B\'ezout domains not satisfying ID$_2$  do not even satisfy GE$_2$, while real quadratic integer rings do satisfy GE$_2$. Proving that some real quadratic integer ring also satisfies property ID$_2$, one would disprove the conjecture and also suggest that the two properties might be equivalent over Dedekind or even Pr\"ufer domains. Recall that Ruitenburg \cite{Ruit} proved the equivalence over B\'ezout domains.
\end{Remark}

\bibliographystyle{plain}

\begin{thebibliography}{10}
\bibitem{AJLL}
A.~Alahmadi, S.~K.~Jain, T.~Y. Lam, A.~Leroy.
\newblock Euclidean pairs and quasi-{E}uclidean rings.
\newblock {\em J. Algebra}, 406: 154--170, 2014.

\bibitem{AJL_1}
A.~Alahmadi, S.~K.~Jain, A.~Leroy.
\newblock Decomposition of singular matrices into idempotents.
\newblock {\em Linear Multilinear Algebra}, 62(1): 13--27, 2014.


\bibitem{BMS}
H.~Bass, J.~Milnor, J.~-P.~Serre
\newblock{Solution of the congruence subgroup problem for ${\rm
  SL}_{n}\,(n\geq 3)$ and ${\rm Sp}_{2n}\,(n\geq 2)$.}
\newblock {\em Inst. Hautes \'Etudes Sci. Publ. Math.}, 33: 59--137, 1967.

\bibitem{BhasRao}
K.~P.~S. Bhaskara~Rao.
\newblock Products of idempotent matrices over integral domains.
\newblock {\em Linear Algebra Appl.}, 430(10): 2690--2695, 2009.

\bibitem{Cohn}
P.~M.~Cohn.
\newblock On the structure of the {${\rm GL}_{2}$} of a ring.
\newblock {\em Inst. Hautes \'Etudes Sci. Publ. Math.}, (30): 5--53, 1966.

\bibitem{CZ}
L.~Cossu, P.~Zanardo.
\newblock{Factorizations into idempotent factors of matrices over Pr\"ufer domains}.
\newblock{\em Communications in Algebra}, 47(4): 1818--1828, 2019.

\bibitem{CZ_PRINC}
L.~Cossu, P.~Zanardo.
\newblock{PRINC domains and comaximal factorization domains}.
\newblock{\em Journal of Algebra and Its Applications}, to appear, 2020.
\newblock{\url{https://www.worldscientific.com/doi/10.1142/S021949882050156X}.}

\bibitem{CZ_Dress}
L.~Cossu, P.~Zanardo.
\newblock{Minimal Pr\"ufer-Dress rings and products of idempotent matrices.}
\newblock {\em Houston J. Math.}, to appear, 2020.
\newblock {\url{https://arxiv.org/abs/1811.09092}.}


\bibitem{CZZ}
L.~Cossu, P.~Zanardo, U.~Zannier.
\newblock Products of elementary matrices and non-{E}uclidean principal ideal
 domains.
\newblock {\em J. Algebra}, 501: 182--205, 2018.

\bibitem{Erdos}
J.~A.~Erdos.
\newblock On products of idempotent matrices.
\newblock {\em Glasgow Math. J.}, 8: 118--122, 1967.

\bibitem{FL}
A.~Facchini, A.~Leroy.
\newblock Elementary matrices and products of idempotents.
\newblock {\em Linear Multilinear Algebra}, 64(10): 1916--1935, 2016.

\bibitem{Fount}
J.~Fountain.
\newblock Products of idempotent integer matrices.
\newblock {\em Math. Proc. Cambridge Philos. Soc.}, 110(3): 431--441, 1991.

\bibitem{HannahOmeara}
J.~Hannah, K.~C.~O'Meara.
\newblock Products of idempotents in regular rings. {II}.
\newblock {\em J. Algebra}, 123(1): 223--239, 1989.

\bibitem{Kap_el_div}
I.~Kaplansky.
\newblock Elementary divisors and modules.
\newblock {\em Trans. Amer. Math. Soc.}, 66: 464--491, 1949.


\bibitem{Laff1}
T.~J.~Laffey.
\newblock Products of idempotent matrices.
\newblock {\em Linear and Multilinear Algebra}, 14(4): 309--314, 1983.

\bibitem{Liehl}
B.~Liehl.
\newblock{On the group SL$_2$ over orders of arithmetic type.}
\newblock{\em J. Reine Angew. Math}, 323: 153--171, 1981.


\bibitem{MRS}
A.~V.~Morgan, A.~S.~Rapinchuk, B.~Sury.
\newblock{Bounded generation of $\rm SL_2$ over rings of $S$-integers with infinitely many units},
\newblock{\em Algebra Number Theory}, 12(8): 1949--1974, 2018.

\bibitem{NZM}
I.~Niven, H.~S.~Zuckerman, H.~L.~Montgomery.
\newblock{\em An introduction to the theory of numbers, fifth edition},
\newblock{John Wiley \& Sons, Inc., New York}, 1991.


\bibitem{Omeara}
O.~T.~O'Meara.
\newblock{ On the finite generation of linear groups over Hasse domains,}
\newblock{\em J. reine angew. Math.}, 217: 79-128, 1965.

\bibitem{PerSalZan}
G.~Peruginelli, L.~Salce, P.~Zanardo.
\newblock {\em Idempotent pairs and {PRINC} domains}, volume 170 of {\em
 Springer Proc. Math. Stat.}, 309--322.
\newblock Springer, [Cham], 2016.

\bibitem{Ruit}
W.~Ruitenburg.
\newblock Products of idempotent matrices over {H}ermite domains.
\newblock {\em Semigroup Forum}, 46(3): 371--378, 1993.

\bibitem{SalZan}
L.~Salce, P.~Zanardo.
\newblock Products of elementary and idempotent matrices over integral domains. \newblock {\em Linear Algebra Appl.}, 452: 130--152, 2014.

\bibitem{Vas}
L.~N.~Vaser\v{s}te\u{\i}n.
\newblock{The group {$SL_{2}$} over {D}edekind rings of arithmetic type.}
\newblock{\em Mat. Sb. (N.S.)}, 89(131): 313--322, 351, 1972.
\end{thebibliography}

\end{document}